\documentclass[12pt]{amsart}
\usepackage{ae,aecompl}

\usepackage[letterpaper]{geometry}
\usepackage{amsthm}
\usepackage{amssymb}
\usepackage{graphicx}

\numberwithin{equation}{section}
\numberwithin{figure}{section}
\theoremstyle{plain}
\newtheorem{thm}{\protect\theoremname}[section]
  \theoremstyle{definition}
  \newtheorem{defn}[thm]{\protect\definitionname}
  \theoremstyle{plain}
  \newtheorem{lem}[thm]{\protect\lemmaname}
  \theoremstyle{remark}
  \newtheorem*{rem*}{\protect\remarkname}
  \theoremstyle{plain}
  \newtheorem{conjecture}[thm]{\protect\conjecturename}
  \theoremstyle{plain}
  \newtheorem{cor}[thm]{\protect\corollaryname}
  \theoremstyle{definition}
  \newtheorem{example}[thm]{\protect\examplename}

  \providecommand{\conjecturename}{Conjecture}
  \providecommand{\corollaryname}{Corollary}
  \providecommand{\definitionname}{Definition}
  \providecommand{\examplename}{Example}
  \providecommand{\lemmaname}{Lemma}
  \providecommand{\remarkname}{Remark}
\providecommand{\theoremname}{Theorem}

\begin{document}

\title{Estimating norms of commutators}

\author{Terry A. Loring}

\author{Fredy Vides}

\address{Department of Mathematics and Statistics, University of New Mexico,
Albuquerque, NM 87131, USA.}

\begin{abstract}
We discuss a general method of finding bounds on the norm of a commutator
of an operator and a function of a normal operator. As an application
we find new bounds on the norm of a commuator with a square root.
\end{abstract}

\maketitle

\section{Norms of Commutators and functional calculus}

{\tt
For later versions of this paper please visit \\
https://repository.unm.edu/handle/1928/23462 \\
UNM Lobo Vault 1928/23462
}

For $f$ a continuous function $\mathbb{R}$ that is periodic, period
$2\pi$ always assumed, then we we will have need to apply if via
functional calculus to both hermitian and unitary elements, in the
latter case by interpreting $f$ as a function on the circle. Just
to be clear, we introduce the notation
\[
f[V]=\tilde{f}(V)
\]
for $V$ any unitary element in a unital $C^{*}$-algebra $\mathcal{A}$,
where
\[
f(z)=f\left(-i\log(z)\right)
\]
 for any $z$ of modulus one. For example
\[
\cos[V]=\tfrac{1}{2}V^{*}+\tfrac{1}{2}V.
\]

It is trivial to prove that when an element $A$ in $\mathcal{A}$
commutes with $V$ then $A$ commutes with $f[V]$. We will need good
estimates that quantify the statement that when $A$ almost commutes
with $V$ then it also almost commutes with $f[V]$. 

The only norm on $\left[A,V\right]$ we really care about is the operator
norm, i.e. the norm on $\mathcal{A}$, that we denote
$\left\Vert \cdot \right\Vert $.  As to functions $f$ that are periodic,
we need
\[
\left\Vert f\right\Vert _{\infty}=\sup_{-\pi\leq x\leq\pi}\left|f(x)\right|
\]
and, whenever $f$ has Fourier series converging absolutely, we use
\[
\bigl\Vert f\bigr\Vert _{F} = \bigl \Vert \hat{f} \bigr \Vert _{1}
\]
the $\ell^{1}$ norm of the Fourier series. We use $\mathcal{U}(A)$
for the group of unitaries in $\mathcal{A}$. 

\begin{defn}
Suppose $f$ is continous and periodic. Define
$\eta_{f}:[0,\infty)\rightarrow[0,\infty)$ by 
\[
f_{\eta}(\delta)
=
\sup\left\{
\left\Vert \left[f[V],A\right]\right\Vert 
\left|\strut\, 
V\in\mathcal{U}(A),\left\Vert A\right\Vert \leq1,\left\Vert \left[V,A\right]\right\Vert \leq\delta
\right.\right\} 
\]
and the supremum is taken over every possible $C^{*}$-algebra $\mathcal{A}$
and taking $V$ and $A$ in $\mathcal{A}$.
\end{defn}

There is a general trend where results about commutators are related to
continuity results involving the functional calculus.  See 
\cite{BhatiaKittanehInequalitiesNormsCommutators}, for example.  In the
case of unitaries there is an easy connection between the two topics.

\begin{lem}
For $f$ that is continous and periodic, if $V$ and $V_{1}$ are unitaries
then
\[
\left\Vert f[V]-f[V_{1}]\right\Vert
\leq
\eta_{f}\left(\left\Vert V-V_{1}\right\Vert \right).
\]
\end{lem}

\begin{proof}
Notice
\[
\left\Vert \left[\left(\begin{array}{cc}
0 & 1\\
1 & 0
\end{array}\right),\left(\begin{array}{cc}
0 & V\\
V_{1} & 0
\end{array}\right)\right]\right\Vert =\left\Vert V-V_{1}\right\Vert 
\]
and
\[
\left\Vert \left[\left(\begin{array}{cc}
0 & 1\\
1 & 0
\end{array}\right),\left(\begin{array}{cc}
0 & f[V]\\
f[V_{1}] & 0
\end{array}\right)\right]\right\Vert =\left\Vert f[V]-f[V_{1}]\right\Vert 
\]
so this is an easy calculation.
\end{proof}

The following generalizes a trick in Pedersen's work on commutators
and square roots, \cite[Lemma 6.2]{PedersenCoronaConstruction}.

\begin{lem}
\label{lem:bound_f_on_circle} Suppose $f$, $g$ and $h$ are continous
and periodic, that $g^{\prime}$ has absolutely convergent Fourier
series. If $f=g+h$ then
\[
\eta_{f}(\delta)\leq m\delta+b
\]
where
\[
m=\left\Vert g^{\prime}\right\Vert _{F}
\]
 and
\[
b=2\min_{\lambda\in\mathbb{C}}\left\Vert h-\lambda\right\Vert _{\infty}.
\]
When $h$ is real valued then $b=\max(h)-\min(h)$.
\end{lem}

\begin{rem*}
In the special case where $h=0$ we recover the folk theorem that
says
\begin{equation}
\left\Vert \left[g[V],A\right]\right\Vert \leq\left\Vert g^{\prime}\right\Vert _{F}\left\Vert \left[V,A\right]\right\Vert \label{eq:commutator_folk_bound}
\end{equation}
for any unitary $V$ and any operator $A$, now without norm restriction
because the two sides are homogeneous in $A$.
\end{rem*}

\begin{proof}
Suppose $\left\Vert A\right\Vert \leq1$ and $V$ is unitary. Since
\[
\left\Vert \left[f[V],A\right]\right\Vert \leq\left\Vert \left[g[V],A\right]\right\Vert +\left\Vert \left[h[V],A\right]\right\Vert 
\]
and
\[
\left\Vert \left[h[V],A\right]\right\Vert =\left\Vert \left[h[V]+\lambda I,A\right]\right\Vert =\left\Vert \left[\left(h+\lambda\right)[V],A\right]\right\Vert 
\]
it suffices to prove equation (\ref{eq:commutator_folk_bound}) and
\begin{equation}
\left\Vert \left[h[V],A\right]\right\Vert 
\leq
2\left\Vert h\right\Vert _{\infty}.
\label{eq:bound_on_h_commutator}
\end{equation}

We know 
\[
g(x)=\sum_{n=-\infty}^{\infty}a_{n}e^{inx}
\]
 where $\sum|na_{n}|<\infty$ and so 
\begin{align*}
\left\Vert \left[g[V],A\right]\right\Vert  & =\left\Vert \left[\sum_{n=-\infty}^{\infty}a_{n}V^{n},A\right]\right\Vert \\
 & \leq\sum_{n=-\infty}^{\infty}\left|a_{n}\right|\left\Vert \left[V^{n},A\right]\right\Vert \\
 & \leq\sum_{n=-\infty}^{\infty}\left|na_{n}\right|\left\Vert \left[V,A\right]\right\Vert \\
 & =\left\Vert g^{\prime}\right\Vert _{F}\left\Vert \left[V,A\right]\right\Vert .
\end{align*}
The spectral theorem tells us $\left\Vert h[V]\right\Vert \leq\left\Vert h\right\Vert _{\infty}$
and so
\[
\left\Vert \left[h[V],A\right]\right\Vert =\left\Vert h[V]A-h[V]A\right\Vert \leq2\left\Vert h[V]\right\Vert \left\Vert A\right\Vert \leq2\left\Vert h\right\Vert _{\infty}.
\]
\end{proof}

As an example, we attack the square root function $f(x)=\sqrt{x}$.
However, this is for $0\leq H\leq1$ replacing $V$ so is about $\gamma_{f}$
not $\eta_{f}$, where $\gamma_{f}$ we now define for working with
functional calculus of postive contractions.

\begin{defn}
Suppose $f$ is continous on $[0,1]$. Define $\eta_{f}:[0,\infty)\rightarrow[0,\infty)$
by 
\[
f_{\eta}(\delta)
=
\sup\left\{ \left\Vert \left[f(H),A\right]\right\Vert \left|\strut\,0\leq H\leq1,\left\Vert A\right\Vert \leq1,\left\Vert H\right\Vert \leq1,\left\Vert \left[H,A\right]\right\Vert \leq\delta\right.\right\} 
\]
and the supremum is taken over every possible $C^{*}$-algebra $\mathcal{A}$
and taking $H$ and $A$ in $\mathcal{A}$. 
\end{defn}

\begin{lem}
\label{lem:bound_f_on_positive} 
Suppose $f$, $g$ and $h$ are continous
on $[0,1]$ and that $g$ is analytic, with power series
\[
g(x)=\sum_{n=0}^{\infty}a_{n}x^n.
\]
If $f=g+h$ then
\[
\eta_{f}(\delta)\leq m\delta+b
\]
where
\[
m=\sum_{n=0}^{\infty}\left|na_{n}\right|
\]
and
\[
b=2\min_{\lambda\in\mathbb{C}}\left\Vert h-\lambda\right\Vert _{\infty}.
\]
\end{lem}

\begin{proof}
We know $\sum|na_{n}|<\infty$ and so 
\begin{align*}
\left\Vert \left[g(H),A\right]\right\Vert  & =\left\Vert \left[\sum_{n=0}^{\infty}a_{n}H^{n},A\right]\right\Vert \\
 & =\left\Vert \sum_{n=0}^{\infty}a_{n}\left[H^{n},A\right]\right\Vert \\
 & \leq\sum_{n=0}^{\infty}\left|a_{n}\right|\left\Vert \left[H^{n},A\right]\right\Vert \\
 & \leq\sum_{n=0}^{\infty}\left|na_{n}\right|\left\Vert \left[H,A\right]\right\Vert .
\end{align*}

It is not clear who first asseted the following, but it appears in
\cite{PedersenCommutatorInequality}.
\end{proof}

\begin{conjecture}
For $f(x)=\sqrt{x}$ we have $\gamma_{f}(\delta)=\sqrt{\delta}$.
Equivalently
\[
\left\Vert \left[H^{\frac{1}{2}},A\right]\right\Vert 
\leq
\left\Vert \left[H,A\right]\right\Vert ^{\frac{1}{2}}
\]
wherever $0\leq H\leq1$ and $\left\Vert A\right\Vert \leq1$.
\end{conjecture}

For any $a$ greater than $0$ and at most $1$ let $g$ be the Taylor
expansion of $f$ at $a$,
\[
g(x)=\frac{1}{2\sqrt{a}}(x-a)+\sqrt{a}.
\]
 and $h=f-g$,
\[
h(x)=\sqrt{x}-\frac{1}{2\sqrt{a}}(x-a)-\sqrt{a}.
\]
 Clearly $\max(h)=h(a)=0$ and the minimum occurs at either $x=0$
or $x=1$, where the values are
\[
h(0)=-\tfrac{1}{2}\sqrt{a}
\]
 and
\[
h(1)=1-\frac{1}{2\sqrt{a}}-\frac{\sqrt{a}}{2}.
\]
For $\tfrac{1}{4}\leq a\leq1$ we find
\[
\min(h)=-\tfrac{1}{2}\sqrt{a}.
\]
Therefore,
\begin{equation}
\eta_{f}(\delta)\leq\frac{1}{2\sqrt{a}}\delta+\tfrac{1}{2}\sqrt{a}
\label{eq:bounds_from_linear_approx}
\end{equation}
for $\tfrac{1}{4}\leq a\leq1$, $ $which is very interesting since
at $\delta=a$ the right hand side is $\sqrt{a}$.  We have
proven a special case of the conjecture, which we state as a lemma.

\begin{lem}
\label{lem:large_delta_root}
When $0\leq H\leq1$ and $\left\Vert A\right\Vert \leq 1$
and $\left\Vert \left[H,A\right]\right\Vert \geq\tfrac{1}{4}$,
we have
\[
\left\Vert \left[H^{\frac{1}{2}},A\right]\right\Vert \leq\left\Vert \left[H,A\right]\right\Vert ^{\frac{1}{2}}.
\]
\end{lem}

Pedersen uses the following easy lemma.

\begin{lem}
If $f_{1}$ is continuous on $[0,1]$ and we set 
\[
f_{2}(x)=1-f_{1}(1-x)
\]
that $\gamma_{f_{1}}=\gamma_{f_{2}}$. 
\end{lem}

His proof of the inequality 
\begin{equation}
\left\Vert \left[H^{\frac{1}{2}},A\right]\right\Vert 
\leq
\frac{2}{\sqrt{\pi}}\left\Vert \left[H,A\right]\right\Vert ^{\frac{1}{2}}\label{eq:gamma_bound_root}
\end{equation}
(notice $2\pi^{\frac{1}{2}}\approx1.128$) in
\cite[Lemma 6.2]{PedersenCoronaConstruction}
invokes Lemma~\ref{lem:bound_f_on_positive} infinitely many times,
as $g$ ranges over the Taylor polynomials for $f(x)=1-\sqrt{1-x}$
exanded at $0$. While (\ref{eq:gamma_bound_root}) is the statement
\[
\gamma_{f}(\delta)\leq\frac{2}{\sqrt{\pi}}\delta^{\frac{1}{2}}
\]
what he actually proves is a bound that is significantly smaller for
$\delta$ close to $1$. Indeed, he showed $\gamma_{f}$ to be bounded
by the function shown in Figure~\ref{fig:GertsBounds}

\begin{figure}
\includegraphics[clip,scale=0.6]{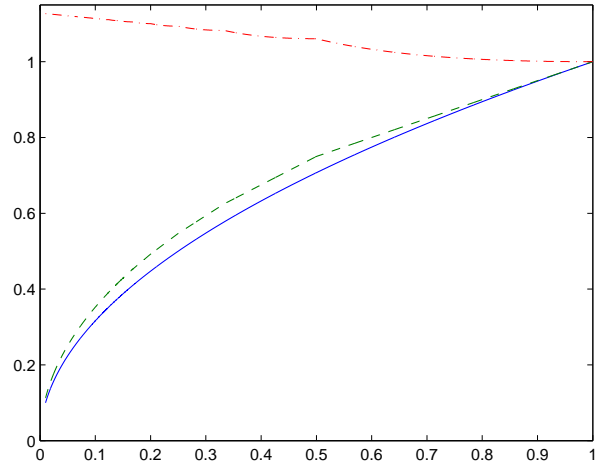}
\caption{
Bound on $\left\Vert \left[H^{\frac{1}{2}},A\right]\right\Vert $
for varying values of $\left\Vert \left[H,A\right]\right\Vert $ as
found by Pedersen, shown as a dashed line. The solid curve is $\sqrt{\delta}$.
The top curve is the ratio of the bound to $\sqrt{\delta}$. 
\label{fig:GertsBounds}
}
\end{figure}

The mininum of all these lines does not lead to an easy formula, so
we state our best theorem regarding the square root in terms
on a ploted function.

\begin{thm}
\label{thm:newRootBound}
If $0\leq H\leq1$ and $\left\Vert A\right\Vert \leq1$ then 
\[
\left\Vert \left[H^{\frac{1}{2}},A\right]\right\Vert 
\leq
\gamma_{0}\left(\left\Vert \left[H,A\right]\right\Vert \right)
\]
where $\gamma$ is the function illustrated in
Figure~\ref{fig:FredysBounds}. 
\end{thm}

\begin{proof}
We simply combine all the linear bounds in \cite[Lemma 6.2]{PedersenCoronaConstruction}
with Lemma~\ref{lem:large_delta_root}.
\end{proof}

\begin{figure}
\includegraphics[clip,scale=0.6]{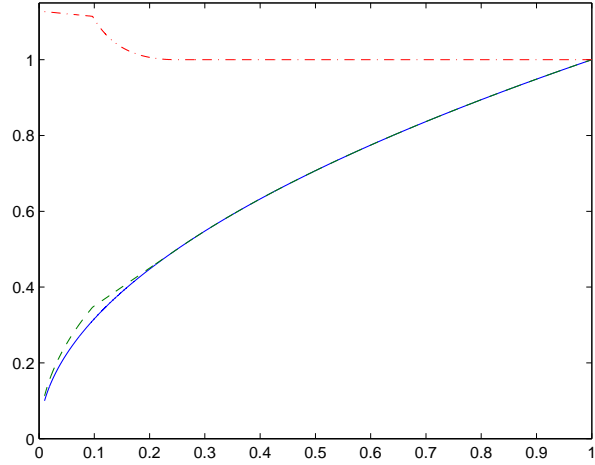} 
\caption{
Bound on $\left\Vert \left[H^{\frac{1}{2}},A\right]\right\Vert $
for varying values of $\delta=\left\Vert \left[H,A\right]\right\Vert $
as improved by the inequalities (\ref{eq:bounds_from_linear_approx}).
The solid curve is $\sqrt{\delta}$.  The dashed curve is the upper
bound $\gamma_0(\delta)$ of Thereom~\ref{thm:newRootBound}.  The top curve
is $\gamma_0(\delta)/\sqrt{\delta}$. 
\label{fig:FredysBounds} 
}
\end{figure}

\section{Examples involving functions on the circle}

There is a desire, driven by investigations in physics
\cite{HastLorTheoryPractice},
to get quantitative results regarding almost commuting matrices. The
Bott index for almost commuting matrices depends on the functional
calculus of unitary matrices. Quantitative studies of the Bott index
require triples of functions
\[
f,g,h:\mathbb{T}^{2}\rightarrow\mathbb{R}^{3}
\]
with certain topological properies. Having a method for dealing with
$\left\Vert \left[f[V],U\right]\right\Vert $for a pair of unitary
elements was the primary motivation for the present paper. 

\begin{cor}
\label{cor:CoeffSumsEstimate}
If $f$ has uniformly converent Fourier
series, 
\[
f(x)=\sum_{n=-\infty}^{\infty}a_{n}e^{inx}
\]
then 
\[
\eta_{f}(\delta)\leq2\sum_{n=-\infty}^{\infty}|a_{n}|.
\]
and
\begin{equation}
\eta_{f}(\delta)
\leq
\delta\sum_{n=-N}^{N}|na_{n}|
+
2\sum_{n=N+1}^{\infty}\left( |a_{n}|+|a_{-n}| \right).
\label{eq:middle_estimate}
\end{equation}
\end{cor}

\begin{proof}
If we take set 
\[
g(x)=\sum_{n=-N}^{N}a_{n}e^{inx}
\]
and apply Lemma~\ref{lem:bound_f_on_circle} we obtain (\ref{eq:middle_estimate}).
For the other we set $g$ to $0$.
\end{proof}

\begin{example}
Consider the triangle wave
\[
f(x)=\begin{cases}
1+\frac{2}{\pi}x & -\pi\leq x\leq0\\
1-\frac{2}{\pi}x & 0\leq x\leq\pi
\end{cases}
\]
we have $a_{2n}=0$ and
\[
a_{2n-1}=\frac{8}{\pi^{2}}\frac{1}{(2n-1)^{2}}.
\]
Using Corollary~\ref{cor:CoeffSumsEstimate} we get the bound on
$\eta_{f}$ as indicated in Figure~\ref{fig:boundingTriangleWave}.
Slighlty better estimates are possible if we
exactly compute the min and max of the difference between $f$ and
its triginometric polynomial approximations. We could also eliminate
the corners by interpolating with trig polynomials between the truncated
Fourier series.
\end{example}

\begin{figure}
\includegraphics[clip,scale=0.6]{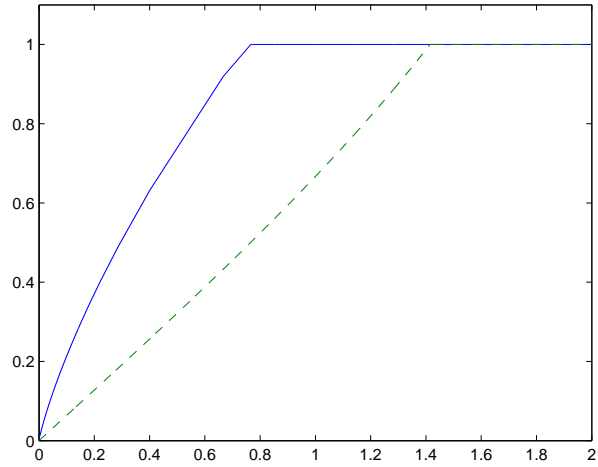}
\caption{
Bounds on $\left\Vert \left[f[V],A\right]\right\Vert $ for varying
values of $\delta=\left\Vert \left[V,A\right]\right\Vert $ for $f$
a triangle wave. The solid line is an upper bound and the dashed line
is a lower bound.
\label{fig:boundingTriangleWave} 
}
\end{figure}

The triangle wave is, up to scaling, the function used in \cite{ExelLoringInvariats}
as one of the functions defining the Bott invariant. To see how well
we are doing in bounding $\left\Vert \left[f[V],A\right]\right\Vert $
we consider a crude lower bound.

\begin{lem}
\label{lem:crude_lower_bound}
If $f$ is periodic and continuous then
for any $\delta<2$ we have 
\[
\eta_{f}(\delta)
\geq
\max
\left\{
		\left|f(x_{2})-f(x_{1})\right| \,
 	\left|\strut\,
   		\left|x_{2}-x_{1}\right|\leq2\arcsin\left(\tfrac{\delta}{2}\right)
   	\right.
\right\} .
\]

\end{lem}
The follows easily from examining the commutator of
\[
\left(\begin{array}{cc}
0 & 1\\
1 & 0
\end{array}\right)
\]
 with
\[
\left(\begin{array}{cc}
e^{ix_{2}} & 0\\
0 & e^{ix_{2}}
\end{array}\right)
\]
 and
\[
f\left[\left(\begin{array}{cc}
e^{ix_{2}} & 0\\
0 & e^{ix_{2}}
\end{array}\right)\right]=\left(\begin{array}{cc}
f(x_{1}) & 0\\
0 & f(x_{2})
\end{array}\right).
\]
 Thus in the example of the triangle wave, we could may have considerable
room to improve our estimate. However, for the purposes of ``quantitative
$K$-theory'' involving the Bott index, this $f$ shows limited potential,
as its companion functions $g$ and $h$ are not so nice. That is,
in the Bott index definition as in \cite{ExelLoringInvariats} we also need
\begin{equation}
\label{eqn:define_h}
h(x)=\begin{cases}
\sqrt{1-\frac{4}{\pi^{2}}x^{2}} & \mbox{ if }x\in\left[-\frac{\pi}{2},\frac{\pi}{2}\right]\\
0 & \mbox{ if }x\notin\left[-\frac{\pi}{2},\frac{\pi}{2}\right]
\end{cases}
\end{equation}
and $\eta_{h}$ tends to zero rather slowly. The crude lower bound
from Lemma~\ref{lem:crude_lower_bound} is shown in
Figure~\ref{fig:boundingBump}.  This is one reason for the switch to
a different triple of functions $f$, $g$ and $h$ in
\cite{HastLorTheoryPractice}.  Commutators involving
the functions in the new and improved Bott index will be
analyized elsewhere.

\begin{figure}
\includegraphics[clip,scale=0.6]{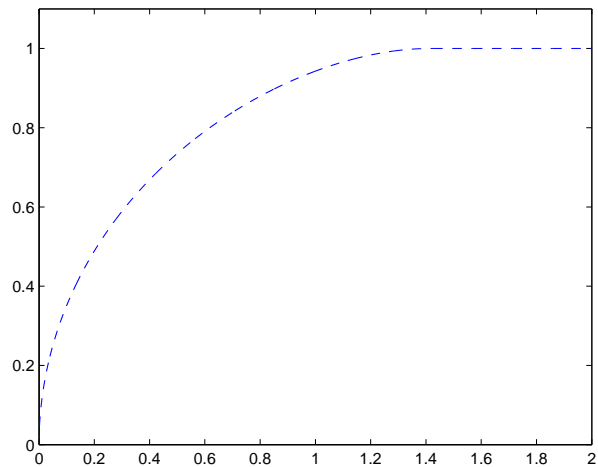}
\caption{A lower bound on
$\left\Vert \left[h[V],A\right]\right\Vert $ where $h$ is the 
bump function on the circle defined in Equation~\ref{eqn:define_h}.
\label{fig:boundingBump} }
\end{figure}

\section{Acknowledgements}

This work was partially supported by a grant from the Simons Foundation
(208723 to Loring).

\rule[0.5ex]{1\columnwidth}{1pt}

\end{document}